\documentclass[11pt,letterpaper]{amsart}
\usepackage[utf8]{inputenc}
\usepackage[russian,english]{babel}
\usepackage{xcolor}
\usepackage{graphicx}
\usepackage{geometry}
\usepackage{amsmath,amsfonts,amssymb, amsthm, url}

\newcommand{\ff}{\mathcal{F}}
\newcommand{\s}{\mathcal{S}}
\newcommand{\g}{\mathcal{G}}
\newcommand{\aaa}{\mathcal{A}}

\newcommand{\cc}{\mathcal{C}}

\newtheorem{thm}{Theorem}

\newtheorem{lem}[thm]{Lemma}

\date{}
\title{Intersection theorems for uniform subfamilies of hereditary families}
\author{Andrey Kupavskii}
\address{G-SCOP, CNRS, University Grenoble-Alpes, France; Moscow Institute of Physics and Technology, Russia, St. Petersburg State University; Email: {\tt kupavskii@ya.ru}}

\begin{document}

\maketitle
\begin{abstract}
  A family $\mathcal C$ of sets is hereditary if whenever $A\in \mathcal C$ and $B\subset A$, we have $B\in \mathcal C$. Chv\'atal conjectured that the largest intersecting subfamily of a hereditary family is the family of all sets containing a fixed element. This is a generalization of the non-uniform Erd\H os--Ko--Rado theorem.
  
  A natural uniform variant of this question, which is essentially a generalization for the uniform Erd\H os--Ko--Rado theorem, was suggested by Borg: given a hereditary family $\mathcal C$, in which all maximal sets have size at least $n$, what is the largest intersecting subfamily of the family of all $k$-element sets in $\mathcal C$? The answer, of course, depends on $n$ and $k$, and Borg conjectured that for $n\ge 2k$ the it is again the family of all $k$-element sets containing a singleton. Borg proved this conjecture for $n\ge k^3$. He also considered a $t$-intersecting variant of the question.
  
  In this paper, we improve the bound on $n$ for both intersecting and $t$-intersecting cases, showing that for $n\ge Ckt\log^2\frac nk$ and $n\ge Ck\log k$ the largest $t$-intersecting subfamily of the $k$-th layer of a hereditary family with maximal sets of size at least $n$ is the family of all sets containing a fixed $t$-element set. We also prove a stability result.
\end{abstract}

\section{Introduction}
In this note, we address a variant of the following classical problem. Assume that $\cc \subset 2^{[N]}$ is a simplicial complex (down-closed family). Let $\ff\subset \cc$ be intersecting. Chv\'atal \cite{Chv} asked for the following generalization of the famous Erd\H os--Ko--Rado theorem: is it true that the largest such family is the family of all sets in $\cc$ that contain some fixed element? In the notation introduced below, is it true that $|\ff|\le |\cc(x)|$ for some element $x\in [N]$?

A natural variant of the question by Chv\'atal was proposed by Borg \cite{Bor3}, in parallel with the uniform Erd\H os--Ko--Rado theorem. For a family $\aaa$, let $\aaa^{(k)}\subset \aaa$ be the subfamily of all $k$-element sets. We say that a simplicial complex has rank $\ge n$ if all  maximal sets in $\mathcal C$ have size at least $n$. Let $\mathcal C$ be a simplicial complex of rank $\ge n$ and  $\ff\subset \cc^{(k)}$ be intersecting. Is it true that $|\ff|\le \cc^{(k)}(x)$ for some $x\in[N]$, provided that $n\ge 2k$? Note that the conjectured bound is the same as in  the uniform Erd\H os--Ko--Rado theorem. Borg proved it for $n>(k-1)(k^2+1)+k$ \cite{Bor3,Bor5}. Another motivation for this question  lies in the fact that it generalizes a question raised by Holroyd and Talbot \cite{HT} on intersecting families of independent sets in graphs. Holroyd and Talbot studied families of  independent $k$-element sets in a fixed graph $G$ and asked, for which $k$ the largest intersecting family of such sets is a star. Holroyd and Talbot conjectured that it is the case if the smallest maximal independent set has size at least $2k$. The connection to the question of Borg can be readily seen if one  reformulates the latter question in terms of intersecting families of $k$-element sets in the independent set complex of $G$.

A natural generalization of this question is to consider  $t$-intersecting families. For ${[n]\choose k}$ we have a theorem due to Frankl \cite{F1} and Wilson \cite{Wil} that the largest $t$-intersecting family $\ff\subset {[n]\choose k}$ is a family of all sets containing a set $T$ of cardinality $t$, provided $n\ge (t+1)(n-t+1)$. Borg \cite{Bor3} conjectured that the analogue of this result holds for any simplicial complex of rank $n$, where $n\ge(t+1)(k-t+1)$.  The validity of the conjecture was shown by Borg \cite{Bor3} for $n\ge (k-t){3k-2t-1\choose t+1}+k$ and later improved by him in \cite{Bor5} to $n\ge (k-t)(k{k\choose t}+1)+k$. (The bound for $t=1$ cited above is a particular case of the latter bound.)

For a family $\aaa\subset 2^{[n]}$, sets $B, X,Y\subset [n]$, where $X\subset Y$, and a family $\s\subset 2^{[n]}$, we will use the following notation:
\begin{align*}
  \aaa(\bar B) = & \{F: F\in \aaa, F\cap B = \emptyset\} \\
  \aaa(B) =  &\{F\setminus B: F\in \aaa, B\subset F\} \\
  \aaa[B] =  &\{F: F\in \aaa, B\subset F\}\\
  \aaa[\s] =  &\bigcup_{B\in \s}\aaa[B]\\
  \aaa(X,Y)= &\{A\setminus X: A \in \aaa, A\cap Y = X\}
\end{align*}
Whenever $B = \{x\}$ for some element $x\in [n]$, we will write
$\aaa(x)$ instead of $\aaa(\{x\})$ etc.

In this note, we prove the following theorem that greatly improves the bounds of Borg and brings them pretty close to the conjectured thresholds.

\begin{thm}\label{thm1}
Let $N,n, k, t$ be positive integers that satisfy the following inequalities:  $n\ge 2^{13}k\log_2(2k)$ and $n\ge 2^{19}tk\log_2^2\frac nk.$   Let $\cc\subset 2^{[N]}$ be a simplicial complex of rank $\ge n$. If $\ff\subset \cc^{(k)}$ is $t$-intersecting then there exists a set $T$ of size $t$ in $[N]$ such that  $|\ff|\le |\cc^{(k)}(T)|$. Moreover, if $m = \min_{T\in {[N]\choose t}} |\ff(T)\setminus \cc^{(k)}(T)|$, then
$$|\ff|\le \max\Big\{\max_{T\in {[N]\choose t}}|\cc^{(k)}(T)|-m2^{2^{-20}\frac n{k\log_2 \frac nk}}, 0.6\max_{T\in {[N]\choose t}}|\cc^{(k)}(T)|\Big\}.$$
\end{thm}
We note that both conditions on $n$ are satisfied if one assumes $n\ge Ctk\log^2k$, where $C$ is some constant. In order to prove this theorem, we use and refine the method of spread approximations developed by Zakharov and the author \cite{KuZa}.

We note that this theorem improves some of the known results on the Holroyd and Talbot conjecture. In particular, it improves the results of Frankl and Hurlbert \cite{FHu} for bounded degree graphs and spider trees.

In the same vein, we can ask many other extremal questions for subfamilies of $\mathcal C$ and $\mathcal C^{(k)}$, such as: what is the largest subfamily avoiding intersection $t-1$? What is the largest family with no $s$-matching? The present method allows for such results.

\section{Preliminaries}

A family $\aaa$ is {\it $r$-spread} if for any set $X$ we have $|\aaa(X)|\le r^{-|X|}|\aaa|$.
Given a family $\aaa\subset 2^{[n]}$ of sets and $q,r\ge 1$, we say that $\aaa$ is {\it $(r,q)$-spread} if for each $S\in {[n]\choose \le q}$, the family $\aaa(S)$ is $r$-spread. Note that putting $S = \emptyset$ implies that $\aaa$ is $r$-spread, so $(r, q)$-spreadness is a stronger condition than the usual $r$-spreadness. A set $W\subset [n]$ is {\it $p$-random} if each element of $[n]$ is included in $W$ independently and with probability $p$.

The following statement is a variant due to Tao \cite{Tao} of the breakthrough result that was proved by Alweiss, Lowett, Wu and Zhang \cite{Alw}.

\begin{thm}[\cite{Alw}, a sharpening due to \cite{Tao}]\label{thmtao}
  If for some $n,r,k\ge 1$ a family $\ff\subset{[n]\le k}$ is $r$-spread and $W$ is an $(m\delta)$-random subset of $[n]$, then $$\Pr[\exists F\in \ff :\ F\subset W]\ge 1-\Big(\frac 5{\log_2(r\delta)} \Big)^mk.$$
\end{thm}

We will need the following slight variant of a statement from the paper of Kupavskii and Zakharov \cite{KuZa}. We give a proof here for completeness.

\begin{thm}\label{thmkz}
 Let $n,k,t\ge1$ be some integers and $\aaa\subset 2^{[n]}$ be an ``ambient'' family. Consider a family $\ff\subset \aaa\cap {[n]\choose \le k}$ that is  $t$-intersecting. Assume that $\aaa$  is $(r_0,t)$-spread. Fix some $q$, and assume that the parameters satisfy the following inequalities: $r_0>r\ge 2q \ge 2t$ and $r> 2^{12}\log_2(2k)$. 

Then there exists  a $t$-intersecting family $\s$ of sets of size at most $q$ and a `remainder' $\ff'\subset \ff$ such that
\begin{itemize}
    \item[(i)] We have $\ff\setminus \ff'\subset \aaa[\s]$;
    \item[(ii)] for any $B\in \s$ there is a family $\ff_B\subset \ff$ such that $\ff_B(B)$ is $r$-spread;
    \item[(iii)] For some $T$ of size $t$ we have $|\ff'|\le r^{q+1}r_0^{t-q-1}|\aaa(T)|$.
  \end{itemize}
\end{thm}

\begin{proof}
Consider the following procedure for $i=1,2,\ldots $ with $\ff^1:=\ff$.
\begin{enumerate}
    \item Find an inclusion-maximal set $S_i$ such that  $|\ff^i(S_i)|\ge  r^{-|S_i|} |\ff^i|$;
    \item If $|S_i|> q$ or $\ff^i = \emptyset$ then stop. Otherwise, put $\ff^{i+1}:=\ff^i\setminus \ff^i[S_i]$.
\end{enumerate}

The family $\ff^i(S_i)$ is $r$-spread. Indeed, arguing indirectly, assume that  there is a set $S_i'$ disjoint from $S_i$ violating this. We have $|\ff^i(S_i\sqcup S_i')|\ge r^{-|S_i'|}|\ff^i(S_i)|\ge r^{-|S_i'|-|S_i|}|\ff^i|$, a contradiction with the maximality of $S_i.$

Let $N$ be the step of the procedure for $\ff$ at which we stop. Put  $\s:=\{S_1,\ldots, S_{N-1}\}$. Clearly, $|S_i|\le q$ for each $i\in [N-1]$. The family $\ff_{B}$ promised in (ii) is defined to be $\ff^i[S_i]$  for $B=S_i$. Next, note that if $\ff^N$ is non-empty, then for any subset $T\subset S_N$ of size $t$ we have $$|\ff^N|\le r^{|S_N|} |\ff^{N}(S_N)|\le  r^{|S_N|} |\aaa(S_N)|\le r^{|S_N|}r_0^{-|S_N|+t}|\aaa(T)|\le  r^{q+1}r_0^{-q-1+t}|\aaa(T)|,$$
where in the second to last inequality we used the fact that $\aaa$ is $(r_0,t)$-spread, and in the last inequality we used $|S_N|\ge q+1$ and $r_0>r$.
We put $\ff':=\ff^m$. Since either $|S_m|>q$ or $\ff' = \emptyset $, and, moreover, $r_0\ge r,$ we have $|\ff'|\le r^{q+1}r_0^{t-q-1}|\aaa(T)|$.

The last and crucial thing to verify is the $t$-intersection property of $\s$.
Take any (not necessarily distinct) $S_1,S_2\in \mathcal S$  and assume that $|S_1\cap S_2|<t$. Recall that for both $\ell\in \{1,2\}$ the family  $\g_\ell:=\ff_{S_\ell}(S_\ell)$ is $r$-spread. We use this in the second inequality below.
  $$|\g_1(\bar S_2)| \ge |\g_1|-\sum_{x\in S_2\setminus S_1} |\g_1(\{x\})|\ge \Big(1-\frac {|S_2|}{r}\Big) |\g_1|\ge \frac 12|\g_1|.$$
  In the last inequality we used that $r\ge 2q.$ The same is valid for $\g_2(\bar S_1)$. Note that both $\g_1':=\g_1(\bar S_2)$ and $\g_2':=\g_2(\bar S_1)$ are subfamilies of $2^{[n]\setminus (S_1\cup S_2)}.$ Moreover,  $\g_\ell'$ is $\frac r2$-spread for both $\ell \in \{1,2\}$.  Indeed, this holds because for any non-empty $Y$ we have $|\g_\ell'(Y)|\le |\g_\ell(Y)|\le r^{-|Y|}|\g_\ell|\le 2r^{-|Y|}|\g'_\ell|\le (r/2)^{-|Y|}|\g'_\ell|$.

  Next, we apply Theorem~\ref{thmtao}. Let us put $m= \log_2(2k)$ and $\delta = (2\log_2(2k))^{-1}$. Note that $m\delta = \frac 12$ and $\frac r2\delta > 2^{10}$ by our choice of $r$.  Theorem~\ref{thmtao} implies that a $\frac{1}{2}$-random subset $W$ of $[n]\setminus (S_i\cup S_j)$ contains a set from $\g_j'$ with probability strictly bigger than
  $$1-\Big(\frac 5{\log_2 2^{10}}\Big)^{\log_2 2k} k = 1-2^{-\log_2 2k} k = \frac 12.$$

  Consider a random partition of $[n]\setminus (S_1\cup S_2)$ into $2$ parts $U_1,U_2$, including each element with probability $1/2$ in each of the parts. Then both $U_\ell$, $\ell\in \{1,2\}$, are distributed as $W$ above. Thus, the probability that there is $F_\ell \in \g_\ell'$, such that $F_\ell\subset U_\ell$, is strictly bigger than $\frac 12$. Using the union bound, we conclude that, with positive probability, it holds that there are such $F_\ell$, $F_\ell\subset U_\ell,$ for each  $\ell \in\{1,2\}$. Fix such a choice of $U_\ell$ and $F_\ell$, $\ell \in \{1,2\}$. But then, on the one hand, both $F_1\cup S_1$ and $F_2\cup S_2$ belong to $\ff$ and, on the other hand, $|(F_1\cup S_1)\cap (F_2\cup S_2)| = |S_1\cap S_2|<t$, a contradiction with $\ff$ being $t$-intersecting.
\end{proof}

\section{Proof of Theorem~\ref{thm1}}
We start by showing that Theorem~\ref{thmkz} is relevant for the case of simplicial complexes.

\begin{lem}\label{lemspread} Let $\mathcal C\subset 2^{[N]}$ be a simplicial complex of rank $\ge n$. Then $\mathcal{C}^{(k)}$ is $\big(\frac{n}{k}, k\big)$-spread.
\end{lem}
\begin{proof} Let us first show that for any $x\in[N]$ we have that $|\cc^{(k)}(x)|\le \frac kn \mathcal{C}^{(k)}$.  In order to do so, we use the local LYM inequality. Consider a bipartite graph $G$ with parts $\mathcal{C}^{(k)}$ and  $\mathcal{C}^{(k)}(x)$. We connect two vertices by an edge if one of the corresponding sets contain the other.

Take an arbitrary vertex of $G$ from the first part. It has degree $1$ in $G$  if the corresponding set $F$ contains $x$ (there is only one set from $\mathcal{C}^{(k)}(x)$ contained in $F$, that is, $F\setminus \{x\}$). Note that there are exactly $|\cc^{(k)}(x)|$ such vertices. Otherwise, it has degree at most $k$ since it has $k$ subsets of size $k-1$. Take an arbitrary vertex $F$ from the second part. There is a maximal set $W$ of size at least $n$ such that $F\subset W$, and thus all sets of the form $F\cup \{w\}$, $w\in W\setminus F$, belong to $\cc^{(k)}$. Thus,  it has degree at least $n-k+1$ in $G$. Double counting the number of edges in the graph, we get
$$
|\mathcal{C}^{(k)}(x)|+k(|\mathcal{C}^{(k)}|-|\mathcal{C}^{(k)}(x)|) \geq (n-k+1)|\mathcal{C}^{(k)}(x)|,
$$
which is equivalent to
$$
\dfrac{|\mathcal{C}^{(k)}(x)|}{|\mathcal{C}^{(k)}|} \leq \dfrac{k}{n}.
$$
Note that for any $Y\subset [N]$ of size $\le k$ the family $\cc(Y)$ is a simplicial complex of rank $\ge n-|Y|$, and the family $\mathcal C^{(k)}(Y)$ is a family of uniformity $k-|Y|$. By what we have proved above, for any $x\in [N]\setminus Y$ we have $$|\mathcal C^{(k)}(Y\cup \{x\})|\le \frac{k-|Y|}{n-|Y|} |\mathcal C^{(k)}(Y)|\le \frac{k}{n} |\mathcal C^{(k)}(Y)|.$$
Applying this inequality iteratively, it is straightforward to see that for any disjoint $X, Y\subset [N]$, where $|Y|\le k,$  we have
$$|\mathcal C^{(k)}(X\cup Y)|\le \Big(\frac{n}{k}\Big)^{|X|} |\mathcal C^{(k)}(Y)|.$$
This shows that the family $\mathcal C^{(k)}$ is $(\frac nk,k)$-spread.
\end{proof}

Let us also recall the following statement from the paper \cite{KuZa}.

Recall that a $t$-intersecting family $\s$ is {\it non-trivial } if $|\cap_{F\in \s} F|<t.$

\begin{thm}\label{thmapproxtrivial}
Let $\varepsilon\in (0,1]$, $n,r_0,q, t \ge 1$  be such that $\varepsilon r_0\ge 2^{17} q \log_2 q$.
Let $\aaa \subset 2^{[n]}$ be an $(r_0, t)$-spread family and let $\s \subset {[n] \choose \le q}$ be a non-trivial $t$-intersecting family.
Then there exists a $t$-element set $T$ such that $|\aaa[\s]| \le \varepsilon |\aaa[T]|$.
\end{thm}
We note that this theorem alone implies the following theorem.
\begin{thm}\label{thmnew1}
    Let $n,k,t>0$ be integers, $k\ge t$. Let $\cc$ be a simplicial complex of rank $n$, and assume that $n\ge 2^{18}k^2\log_2k.$  Take the largest $t$-intersecting family $\ff\subset \cc^{(k)}$. Then there exists a set $T$ of size $t$ such that $\ff= \cc^{(k)}(T)$. Moreover, if $\ff$ is non-trivial then $|\ff|\le \frac 12|\cc^{(k)}(T)|.$
\end{thm}
\begin{proof}
    We directly apply Theorem~\ref{thmapproxtrivial}. Put $\varepsilon=\frac 12$ and take $\s  = \ff.$ The family $\cc^{(k)}$ is $(\frac nk,k)$-spread, and so the conditions of Theorem~\ref{thmapproxtrivial} are fulfilled with $q = k$. If $\ff$ is non-trivial then we conclude that it is at most $\frac 12|\aaa(T)|$ for some $t$-element $T.$ Thus, $\ff$ is trivial and so must coincide with $\aaa(T)$ for some $t$-element $T$.
\end{proof}
We will need the following lemma that relates the sizes of complexes and its restrictions.
\begin{lem}\label{lem4}
Fix some integers $n,s,k,t>0$.    Let $\cc$ be a simplicial complex of rank $n$, $T$ be a subset of size $t$, and $F$ a set of size $s$. Then $|\cc^{(k)}(T,F\cup T)|\ge \Big(1-\frac{k-t}{n-s-t}\Big)^s|\cc^{(k)}(T)|$.
\end{lem}
\begin{proof}
    Put $F = \{x_1,\ldots, x_s\}$, $T_i = \{x_1,\ldots, x_i\}$ for $i=0,\ldots, s$, and let us prove the following inequality for each $i\in [s]:$
    \begin{equation}\label{eq12}
        |\cc^{(k)}(T,F\cup T_i)|\ge \Big(1-\frac{k-t}{n-s-t}\Big)|\cc^{(k)}(T,F\cup T_{i-1})|.
    \end{equation}
In order to prove this, put $\aaa:= \cc^{(k)}(T,F\cup T_{i-1})$ and $\cc':= \cc(T,F\cup T_{i-1})$. Note that $\cc'$ is a simplicial complex with rank at least $n-s-t$, and that the sets in $\aaa$ have all size $k-t$. Lemma~\ref{lemspread} implies that $\aaa$ is $\frac{n-s-t}{k-t}$-spread. Then the left-hand side of \eqref{eq12} is  $|\aaa|-|\aaa(x_i)|,$ which is at least $(1-\frac{k-t}{n-s-t})|\aaa|$ by spreadness.

Combining inequalities \eqref{eq12} for all $i\in [s]$, we get the statement of the lemma.
\end{proof}

We are now ready to prove Theorem~\ref{thm1}.
\begin{proof}[Proof of Theorem~\ref{thm1}]
In view of Theorem~\ref{thmnew1}, we may assume that \begin{equation}\label{eq11}
    n<2^{18}k^2\log_2k.
\end{equation}

    Let us first fix the parameters to satisfy all the conditions in Theorems~\ref{thmapproxtrivial},~\ref{thmkz}. We put $\varepsilon=\frac 12$, $r_0 = \frac nk$ and take $q = 2^{-18}\frac n{k\log_2\frac nk}.$ We note that $q< k$ in view of \eqref{eq11}. Indeed, $q$ is monotone increasing as $n$ increases (for fixed $k$ and given lower bounds on $n$), and so it is sufficient to verify for $n = 2^{18}k^2\log_2 k$. But then $q = \frac {k\log_2 k}{\log_2(2^{18} k\log_2 k)}<k.$
   Note that $\cc^{(k)}$ is $(r_0, q)$-spread and that $q\log_2 q \le 2^{-18}\frac nk = 2^{-17}\varepsilon r_0.$ This verifies the conditions of Theorem~\ref{thmapproxtrivial}.

    In terms of Theorem~\ref{thmkz}, we put $r = r_0/2$ and note that $r_0/2 = r>2q$ by the above. We also have $q\ge t$ because $k>t$ and $n\ge 2^{19}tk\log_2^2\frac nk.$ Finally, we have $r = r_0/2 = \frac n{2k}>2^{12}\log_2(2k).$ Thus, all the conditions of Theorem~\ref{thmkz} are verified.

    Next, let us see what does these theorems imply in our situation. Theorem~\ref{thmkz} gives us a $t$-intersecting family $\s$ of sets of size at most $q$ and a remainder $\ff'$ of size
    \begin{multline}\label{eq13}
      |\ff'|\le r^{q+1}r_0^{t-q-1}|\cc^{(k)}(T)| = 2^{-q-1}(n/k)^t |\cc^{(k)}(T)| = \\
      2^{-2^{-18}\frac n{k\log_2 \frac nk}+t\log_2 \frac nk}|\cc^{(k)}(T)| \le 2^{-2^{-19}\frac n{k\log_2 \frac nk}}|\cc^{(k)}(T)|. 
      \end{multline}
In particular, it is easy to check that $|\ff'|<0.1 |\cc^{(k)}(T)|.$ Next, apply Theorem~\ref{thmapproxtrivial} to $\mathcal S$. We get that for some $t$-element set $T'$ we have $|\cc^{(k)}(\mathcal S)|\le 0.5|\cc^{(k)}(T')|$ unless $\s$ is trivial, i.e., $\s = \{T\}$ for some $t$-element $T$. If $\s$ is non-trivial, then $|\ff|\le |\cc^{(k)}(\s)|+|\ff'|\le 0.6 |\cc^{(k)}(T)|$ for some $t$-element $T$. In what follows, we suppose that $\s$ consists of a single set $T$.

This argument alone is sufficient to prove an approximate stability result, stating that $\ff$ that is close to maximal must be ``almost contained'' in a family $\cc^{(k)}[T]$. In what follows, we will make better use of the properties of the family $\ff'$ in order to prove the exact result and stability.

 First, we start by removing all sets from $\ff'$ that contain $T$. The remaining family, which we denote $\ff''$, satisfies $\ff''=\ff\setminus \cc^{(k)}[T]$. In what follows, we assume that $\ff''$ is non-empty (otherwise, we are done). We need to show that $|\ff''|<  |\cc^{(k)}[T]\setminus \ff|$.   Note that $|\ff''|\le |\ff'|$, and thus the bound \eqref{eq13} is valid for $|\ff''|$ as well.

  The next argument borrows some of the ideas from  the proof of Theorem~\ref{thmkz}. Since $\ff''$ is non-empty, we can find a maximal subset $U$ of the ground set that satisfies $|\ff''(U)|\ge r^{-|U|}|\ff''|$. Note that the family $\ff''(U)$ is $r$-spread. Informally, our aim is to show that
  $$|\cc^{(k)}(T,T\cup U)\setminus \ff(T,T\cup U)|\gg |\ff''|.$$
  In order to do so, we proceed in two steps. The first step is to show that   $|\cc^{(k)}(T,T\cup U)|\gg |\ff''|.$ The second step is to show that a large portion of the sets in   $\cc^{(k)}(T,T\cup U)$ cannot belong to $ \ff(T,T\cup U).$

Let us do the first step. We have to consider two cases. The first is that $|U|\le q$. In this case, we use \eqref{eq13} as an upper bound on $|\ff''|$. At the same time, Lemma~\ref{lem4} implies $$|\cc^{(k)}(T,U\cup T)|\ge \Big(1-\frac{k-t}{n-|U|-t}\Big)^{|U|}|\cc^{(k)}(T)|\ge \Big(1-\frac{k-t}{n-q-t}\Big)^{q}|\cc^{(k)}(T)|\ge \frac 12|\cc^{(k)}(T)|,$$
where the last inequality is due to our choice of parameters.
Thus, we get that
\begin{equation}\label{eq14}\frac{|\cc^{(k)}(T,T\cup U)|}{|\ff''|}\ge 2^{2^{-19}\frac n{k\log_2 \frac nk}-1}.\end{equation}
If $|U|\ge q+1$ then we have the following analogue of \eqref{eq13}:
\begin{multline}\label{eq15}|\ff''|\le r^{|U|}r_0^{t-|U|}|\cc^{(k)}(T)| =  2^{-|U|+q+1} r^{q+1}r_0^{t-q-1} |\cc^{(k)}(T)|\\ \le 2^{-|U|+q+1}  2^{-2^{-19}\frac n{k\log_2 \frac nk}}|\cc^{(k)}(T)|.\end{multline}
At the same time, Lemma~\ref{lem4} implies \begin{multline*}|\cc^{(k)}(T,U\cup T)|\ge \Big(1-\frac{k-t}{n-|U|-t}\Big)^{|U|}|\cc^{(k)}(T)|\ge \\ \Big(1-\frac{k-t}{n-|U|-t}\Big)^{q+1}\cdot \Big(1-\frac{k-t}{n-|U|-t}\Big)^{|U|-q-1}|\cc^{(k)}(T)|\ge \frac 12\cdot 2^{(-|U|+q+1)/2}|\cc^{(k)}(T)|.\end{multline*}
(Note that the last inequality is very loose.) Thus, we can see that \eqref{eq14} holds for $|U|\ge q+1$ as well.

We go on to the second step. First, consider the family $\ff''(U)$. It is $\frac{n}{2k}$-spread. Then, $\ff''(U,T\cup U)$ is $(\frac{n}{2k}-t)$-spread (cf. the proof of Theorem~\ref{thmkz}). The families $\ff''(U, T\cup U)$ and $\ff(T, T\cup U)$ are cross-intersecting.
Assume that $|\ff(T, T\cup U)|\ge \frac 12 |\cc^{(k)}(T,T\cup U)|$. Note that  $\cc(T,T\cup U)$ is a simplicial complex of rank at least $n-(t+|U|)\ge n-t-k$, and that $\cc^{(k)}(T,T\cup U)$ consists of all sets of size $k-t$ from that complex. Thus, by Lemma~\ref{lemspread}, $\cc^{(k)}(T,T\cup U)$ is $\frac{n-k-t}{k-t}$-spread. From here and the assumption on  the size of $\ff(T, T\cup U)$, we get that $\ff(T, T\cup U)$ is $\frac{n-k-t}{2(k-t)}$-spread.    We use Theorem~\ref{thmtao} with $m = \log_2(2k), \delta = \frac 1{2\log_2(2k)},$ and produce a random $2$-coloring of  $[N]\setminus (T\cup U)$. As in the proof of Theorem~\ref{thmkz}, we conclude that with probability $>1/2$ the first color contains a set from $\ff(T, T\cup U)$ and with probability $>1/2$ the second color contains a set from $\ff''(U, T\cup U)$. Together, this contradicts the cross-intersecting property. (In order for these two properties to hold, we need the spreadness of both families to be large, concretely, we need  $\frac n{2k}-t> 2^{10}\delta^{-1} = 2^{11}\log_2(2k)$ and $\frac{n-k-t}{2(k-t)}>2^{11}\log_2(2k)$. Both inequalities are valid in our assumptions on $n$.)

That is, $|\ff(T, T\cup U)|< \frac 12 |\cc^{(k)}(T,T\cup U)|,$ and, combining with \eqref{eq14} we get that
\begin{equation}\label{eq16}\frac{|\cc^{(k)}(T,T\cup U)\setminus \ff(T,T\cup U)|}{|\ff''|} \ge \frac{|\cc^{(k)}(T,T\cup U)|}{2|\ff''|}\ge 2^{2^{-19}\frac n{k\log_2 \frac nk}-2}> 1.\end{equation}
Denote $m:=|\ff''|$. Then by \eqref{eq14} $|\ff|\le |\cc^{(k)}(T)|+m - m2^{2^{-19}\frac n{k\log_2 \frac nk}-2}\le |\cc^{(k)}(T)|- m2^{2^{-20}\frac n{k\log_2 \frac nk}}$. The last inequality is valid due to our choice of $n$.  This concludes the proof of the main part and the stability part of the theorem.
\end{proof}

\end{document}